\documentclass{article}
\usepackage{romp}

\usepackage{amsmath,amsthm}
\usepackage{amsfonts}
\usepackage{amssymb}

\makeatletter
\newcommand\Label[1]{&\refstepcounter{equation}(\theequation)\ltx@label{#1}&}
\makeatother

\title{ Double-graded quantum superplane }
\author{ Andrew James Bruce \\ Mathematics Research Unit, University of Luxembourg\\
Maison du Nombre, L-4364 Esch-sur-Alzette, Luxembourg \\ e-mail: andrewjamesbruce@googlemail.com \\[2ex]
          Steven Duplij
                      \\ Center for Information Technology (WWU IT), Universit\"at M\"unster,\\ D-48149 M\"unster, Deutschland \\ e-mail: douplii@uni-muenster.de }

\begin{document}

\maketitle
\begin{abstract}
A $\mathbb{Z}_2 \times \mathbb{Z}_2$-graded generalisation of  the quantum superplane is proposed and studied. We construct a bicovariant calculus on what we shall refer to  as the \emph{double-graded quantum superplane}. The commutation rules between the coordinates, their differentials and partial derivatives are explicitly given. Furthermore, we show that an extended version of the double-graded quantum superplane admits a natural Hopf $\mathbb{Z}_2^2$-algebra structure.
\end{abstract}


\section{Introduction}
Noncommutative geometry has  been playing an ever-increasing r\^{o}le in mathematics and physics over the past few decades (see for example \cite{con94,dup/sie/bag,man91}). At the scale at which quantum effects of the gravitational field are dominant, it is expected that space-time will depart from its classical smooth Riemannian structure. Upon rather general arguments, space-time is expected to be some kind of noncommutative geometry. Unfortunately, nature has so far provided few hints as to what one should expect from these generalised geometries. The fundamental objects at play here are associative algebras and differential calculi over them. Woronowicz \cite{woro89} initiated the study of quantum groups and their differential calculi as the basic objects in noncommutative geometry. This approach stresses that the properties of the quantum group are key to constructing differential calculi. A different approach follows Manin's philosophy (see \cite{man89})  that differential forms on noncommutative spaces are defined in terms of their noncommutative or quantum coordinates and the properties of quantum groups acting upon these spaces.  Wess and Zumino \cite{wes/zum91} used the approach of Manin to define a covariant differential calculus on the quantum hyperplane. The first description of Manin's quantum plane as a Hopf algebra is by Tahri \cite{tah98}. For quantisations of various superspaces and their corresponding differential calculi see \cite{bre93,cel16,fal/tah01,oza13,son91}. Nontrivial actions of quantized universal enveloping algebras on the quantum plane were considered in \cite{dup/sin3}. We remark that quantum groups (Hopf algebras), due to their tight relation with the Yang--Baxter equation, are important in conformal field theory, statistical mechanics, integrable systems, etc. Indeed, quantum groups, as a particular class of Hopf algebras, originated in the work of  Drinfel'd  and Jimbo (see \cite{dri87}) on quantum inverse scattering.  Today it is realised that many combinatorial aspects of physics have neat formulations in terms of Hopf algebras \cite{duc/bla/hor/pen/sol}.    \par
Inspired by the recently developed locally ringed space approach to $\mathbb{Z}_2^n$-manifolds (see \cite{bru/iba/pon19,bru/pon19b,bru/pon19a,Cov/Gra/Pon1,Cov/Gra/Pon2}), we examine  quantum $\mathbb{Z}_2^2$-planes, or as we prefer to call them, \emph{double-graded quantum superplanes}.   Such noncommutative geometries are the simplest examples of  noncommutative $\mathbb{Z}_2^n$-spaces $(n\geq 2)$. Much like supermanifolds, $\mathbb{Z}_2^n$-manifolds offer a `halfway house' on one's passage from classical geometry to noncommutative geometry.  Quantising superspaces and similar offers a deeper picture here as well as very workable examples of noncommutative geometries. Indeed, `non-anticommuting superspaces' have long been studied in physics because various background fields in string theory lead to noncommutative deformation of superspace.   For example, R-R field backgrounds lead to `$\theta-\theta$' deformations and gravitino backgrounds lead to  `$x-\theta$' deformations (see \cite{bo/gra/nie, sei03}). It is probably fair to say that the mathematics literature on `noncommutative superspaces' is not so developed (the reader may consult \cite{deG:2015} for an overview).  \par
We also point out that $\mathbb{Z}_2^n$-geometry, as well as the double-graded quantum superplane  sit comfortably within Majid's framework of braided geometry, see \cite{maj93,maj95,maj96} for  very accessible reviews. The general idea is to replace the standard bose-fermi sign factors with a more general braided relation.  Noncommutative geometry formulated in braided monoidal categories has been very successful and we must mention toric noncommutative geometry as a key example, see for example \cite{brn/sch/sza,bra/lan/sui,cir/lan/sza}. In this sense, we present a very specific example of a braided geometry inspired by $\mathbb{Z}_2^n$-geometry.  Loosely, a $\mathbb{Z}_2^n$-manifold is a ``manifold'' with coordinates that are assigned a degree in $\mathbb{Z}_2^n$ and their sign rule under exchange is given  in terms of the standard scalar product of their degrees. The geometry we study in this paper is a $q$-deformed version of the $\mathbb{Z}_2^2$-plane, which can be understood as the algebra $C^\infty(\mathbb{R})[[\xi, \theta, z]]$ subject to $\xi^2 =0$, $\theta^2 =0$, $\xi \theta = \theta \xi$, $\xi z = - z \xi$  and $\theta z =-  z \theta$.  The coordinate $x$ on $\mathbb{R}$ strictly commutes with everything.  Note that these relations are not super, i.e., not simply $\mathbb{Z}_2$-graded commutative. For example, if we assign degree $1$ to $\xi$ and $\theta$ then their commutation rule is not fully determined by their degrees. Furthermore, assigning degree $0$ or $1$ to $z$ leads to the same conclusion. Note that $z$ is not taken to be nilpotent. However, assigning degree $(0,0)$,$(0,1)$, $(1,0)$ and $(1,1)$ to $x$, $\xi$, $\theta$ and $z$, respectively, and then defining the commutation factor to be
$$a     b = (-1)^{\langle \textnormal{deg}(a), \textnormal{deg}(b) \rangle}     b     a,$$
where $a,b \in \{x, \xi, \theta, z \}$ and $\langle - , - \rangle$ is the standard scalar product, reproduces exactly the desired commutation rules. We must also mention the related notion of paragrassmann variables $\psi$, where $\psi^p = 0$ for some $p> 2$ (see for example \cite{fil/isa/kur92, fil/isa/kur93,ger/tka,ger/tka85}). The relation between paragrassmann variables, parastatistics and $\mathbb{Z}_2^n$-graded commutative algebras is via Green's  ansatz (see \cite{Gre,Vol})
 \par
The readers  attention should be brought to the fact that Scheunert proved a theorem reducing ``coloured'' Lie algebras to either Lie algebras or Lie superalgebras \cite{Sch}. Similarly, Neklyudova proved an analogue of this theorem for $\mathbb{Z}_2^n$-graded,  graded-commutative, associative algebras \cite{leites13}. Neither of these theorems rules out the study of $\mathbb{Z}_2^n$-geometry. One deals with very specific algebras when studying $\mathbb{Z}_2^n$-manifolds, for instance, and quite often trying to ``pullback''  supergeometric constructions to the category of $\mathbb{Z}_2^n$-manifolds is non-trivial.   The study of $\mathbb{Z}_2^n$-manifolds does not reduce to the study of supermanifolds. Moreover, in this paper, we study a particular $\mathbb{Z}_2^2$-graded, associative algebra that is not graded-commutative. Neklyudova's theorem does not directly apply here. \par
 In section \ref{sec:pre} we define Hopf $\mathbb{Z}_2^2$-algebras and bicovariant differential calculi on them.  There are no truly new results in this section.  Indeed, the earliest reference we are aware of to the notion of a $G$-graded (here $G$ is an abelian group) or coloured Hopf algebra is \cite[Definition 10.5.11]{mon93}. Moving on to section  \ref{sec:QuantDoubSupPlane}, we present the double-graded quantum superplane $\mathbb{R}_q(1|1,1,1) =: \mathbb{R}_q(1|\mathbf{1}) $ as a quantisation of the $\mathbb{Z}_2^2$-plane with a single parameter, which we denote as $q$. We explore the $\mathbb{Z}_2^2$-bialgebra structure on such `spaces'. The   Hopf $\mathbb{Z}_2^2$-algebra structure on an extended version is also given. We explicitly construct a bicovariant differential calculus on $\mathbb{R}_q(1|\mathbf{1})$. Moreover, we deduce all the required commutation relations between the generators of the algebra, their differentials and their partial derivatives. The resulting structures closely resemble two copies of Manin's quantum superplane $\mathbb{R}_q(1|1)$ \cite{man89}, but with subtle interesting differences due to our underlying $\mathbb{Z}_2^2$-grading.  This needs to be kept in mind in order to understand the appearance of various signs in the commutation relations, as well as when dealing with the tensor product - we will always use the $\mathbb{Z}_2^2$-graded tensor product.  We will present all relevant calculations explicitly for clarity and accessibility. We end in section \ref{sec:ConRem} with a few concluding remarks.

\smallskip

\noindent \textbf{Conventions and Notation}: We work over the field $\mathbb{C}$  and we set $\mathbb{Z}_2^n := \mathbb{Z}_2 \times \mathbb{Z}_2 \times \cdots \times \mathbb{Z}_2$ ($n$-times). In particular, $\mathbb{Z}_2^2 := \mathbb{Z}_2 \times \mathbb{Z}_2$. We fix the order of elements in $\mathbb{Z}_2^2$  \emph{lexicographically}, i.e.,
$$\mathbb{Z}_{2}^{2} = \{ (0,0),  \: (0,1), \: (1,0), \: (1,1)\}   .$$
Note that other choices of ordering have appeared in the literature. We will denote elements of $\mathbb{Z}_2^2$ by $\gamma_i$, understanding that $i = 0,1,2,3$ using the above fixed ordering.  The abelian group $\mathbb{Z}_2^2$ comes with a canonical scalar product that we will denote as $\langle -, - \rangle$. In particular, setting $\gamma_i = (a,b)$ and $\gamma_j = (a',b')$, we have $\langle \gamma_i , \gamma_j \rangle = a a' + b b'$. The generalisation to $\mathbb{Z}_2^n$ $(n>2)$ is clear.

\section{Preliminaries}\label{sec:pre}
\subsection{Hopf $\mathbb{Z}_2^2$-algebras}
Standard references for Hopf algebras and their application in noncommutative geometry include \cite{cha/pre95,gra/var/fig01,mad99, man91}. The notion of  ``coloured Hopf algebras'' is not well-known but has appeared in the literature over the years, see for example \cite{and/ang/bag14,fel01,hur/mak17,mon93,wang11}. Rather than define quite general structures, we will work with the specific example of coloured Hopf algebras that have an underlying $\mathbb{Z}_2^2$-graded structure. The generalisation to $\mathbb{Z}_2^n$-graded structure can be made verbatim only making minimal changes.
\begin{definition}{Definition}
A \emph{Hopf $\mathbb{Z}_2^2$-algebra} is a Hopf algebra in the category of $\mathbb{Z}_2^2$-graded vector spaces.
\end{definition}
While the above definition is complete, we will spell-out the structure of a Hopf $\mathbb{Z}_2^2$-algebra piece-by-piece for clarity.  Note that the tensor product of $\mathbb{Z}_2^2$-algebras is the $\mathbb{Z}_2^2$-graded tensor product, i.e.,
$$(a\otimes b)(c \otimes d) = (-1)^{\langle \textnormal{deg}(b),  \textnormal{deg}(c) \rangle }     ac \otimes bd.$$
\begin{definition}{Definition}
A \emph{$\mathbb{Z}_2^2$-algebra} is a triple $(\mathcal{A} , \mu , \eta)$, where $\mathcal{A} =  \underset{\gamma_i \in \mathbb{Z}_2^2}{\oplus} \mathcal{A}_{\gamma_i}$ is a $\mathbb{Z}_2^2$-graded vector space, $\mu : \mathcal{A} \otimes_{\mathbb{C}} \mathcal{A} \rightarrow \mathcal{A}$ (\emph{multiplication}) and $\eta : \mathbb{C} \rightarrow  \mathcal{A}$ (\emph{unit}) are two (grading preserving) $\mathbb{Z}_2^2$-graded space morphisms  that satisfy
\begin{align}
& \mu \circ (\mu \otimes \textnormal{Id}_{\mathcal{A}}) = \mu \circ ( \textnormal{Id}_{\mathcal{A}} \otimes \mu) && \textnormal{(Associativity)}\\
& \mu \circ (\eta \otimes \textnormal{Id}_{\mathcal{A}}) = \mu \circ( \eta\otimes \textnormal{Id}_{\mathcal{A}})= \textnormal{Id}_{\mathcal{A}} && \textnormal{(Unity)}
\end{align}
 where we  have used the natural isomorphisms $\mathbb{C}\otimes \mathcal{C} \cong \mathcal{C} \cong \mathcal{C} \otimes \mathbb{C}$.  A map $\phi : \mathcal{A} \rightarrow \mathcal{B}$ is a \emph{$\mathbb{Z}_2^2$-algebra morphism} if it is a (grading preserving) $\mathbb{Z}_2^2$-graded space morphism that satisfies
\begin{align}
\phi \circ \mu_{\mathcal{A}} = \mu_{\mathcal{B}} \circ \phi \otimes \phi, &&\textnormal{and} && \phi \circ \eta_{\mathcal{A} } = \eta_{\mathcal{B}}.
\end{align}
\end{definition}
\begin{definition}{Definition}\label{def:coalgebra}
A \emph{$\mathbb{Z}_2^2$-coalgebra} is a triple $(\mathcal{C}, \Delta, \epsilon)$, where $\mathcal{C}$ is a $\mathbb{Z}_2^2$-graded vector space, $\Delta : \mathcal{C} \rightarrow \mathcal{C} \otimes \mathcal{C}$ (coproduct) and $\epsilon : \mathcal{C} \rightarrow \mathbb{C}$ (counit) are two (grading preserving) $\mathbb{Z}_2^2$-graded space morphisms  that satisfy
\begin{align}
& (\Delta \otimes \textnormal{Id}_{\mathcal{C}}) \circ \Delta = (\textnormal{Id}_{\mathcal{C}} \otimes \Delta)\circ \Delta  && \textnormal{(Coassociativity)}\\
& (\epsilon \otimes \textnormal{Id}_{\mathcal{C}} )\circ \Delta = (\textnormal{Id}_{\mathcal{C}} \otimes \epsilon) \circ \Delta = \textnormal{Id}_{\mathcal C}&& \textnormal{(Counity)}
\end{align}
where we  have used the natural isomorphisms $\mathbb{C}\otimes \mathcal{C} \cong \mathcal{C} \cong \mathcal{C} \otimes \mathbb{C}$ in the last equality of counity condition. A map $\phi : \mathcal{C} \rightarrow  \mathcal{D}$ is a \emph{$\mathbb{Z}_2^2$-coalgebra morphism} if it is a (grading preserving) $\mathbb{Z}_2^2$-graded space morphism that satisfies
\begin{align}
\phi \otimes \phi \circ \Delta_{\mathcal{C}} = \Delta_{\mathcal{D}} \circ \phi , &&\textnormal{and} && \epsilon_{\mathcal{D} } \circ \phi = \epsilon_{\mathcal{C}}.
\end{align}
\end{definition}
\begin{definition}{Definition}\label{def:bialgebra}
A \emph{$\mathbb{Z}_2^2$-bialgebra} is a tuple $(\mathcal{A}, \mu, \eta, \Delta, \epsilon)$ where $(\mathcal{A}, \mu, \eta)$ is a $\mathbb{Z}_2^2$-algebra and $(\mathcal{A}, \Delta, \epsilon)$ is a $\mathbb{Z}_2^2$-coalgebra such that the following equivalent compatibility conditions hold
\begin{enumerate}
\item $\Delta :  \mathcal{A} \rightarrow \mathcal{A} \otimes \mathcal{A}$ and $\epsilon : \mathcal{A} \rightarrow \mathbb{C}$ are $\mathbb{Z}_2^2$-algebras morphisms,
\item $\mu : \mathcal{A} \times \mathcal{A} \rightarrow \mathcal{A}$ and $\eta : \mathbb{C} \rightarrow \mathcal{A}$ are $\mathbb{Z}_2^2$-coalgebra morphisms.
 \end{enumerate}
 A \emph{morphism of $\mathbb{Z}_2^2$-bialgebras} is a (grading preserving) $\mathbb{Z}_2^2$-graded space morphism that is both a morphism of $\mathbb{Z}_2^2$-algebras and $\mathbb{Z}_2^2$-coalgebras.
\end{definition}

\begin{proposition}{Proposition}\label{prop:Z22counit}
Let $(\mathcal{A}, \mu, \eta, \Delta, \epsilon)$  be a $\mathbb{Z}_2^2$-bialgebra with unit element $\eta(1) = 1 \ni \mathcal{A}_{\gamma_0}$. Then
$\epsilon(\mathcal{A}_{\gamma_i}) =0 $ for all $1 \leq i \leq 3$, and $\epsilon(1) = 1$.
\end{proposition}
\begin{proof}
As $\epsilon : \mathcal{A} \rightarrow \mathbb{C}$ is a grading preserving map it is clear that $\epsilon(\mathcal{A}_{\gamma_i}) =0$ with the exception of $i = 0$, i.e.,  $\epsilon(\mathcal{A}_{(0,0)})$  cannot be zero as $\epsilon$ is required to be a morphism of (unital) algebras, thus the unit in $\mathcal{A}$ must be sent to the unit in $\mathbb{C}$, i.e., the number $1$.
\end{proof}
\begin{definition}{Definition}
Let $\mathcal{A}$ be a $\mathbb{Z}_2^2$-bialgebra. Then $a \in \mathcal{A}_{(0,0)}$ is said to be a \emph{group-like element} if $\Delta(a) = a \otimes a $. An element $b \in \mathcal{A}$ is said to be a \emph{primitive element}  if  $\Delta(b) = b \otimes 1 + 1 \otimes b$.
\end{definition}
\begin{proposition}{Proposition}
The set of primitive elements of a $\mathbb{Z}_2^2$-bialgebra form a $\mathbb{Z}_2^2$-Lie algebra under the $\mathbb{Z}_2^2$-graded commutator.
\end{proposition}
\begin{proof}
As the coproduct is a linear map, it is sufficient to consider homogeneous primitive elements and show that they are closed under the $\mathbb{Z}_2^2$-graded commutator. A direct calculation shows that
\begin{align}
\nonumber \Delta([a,b])  & = \Delta(ab - (-1)^{\langle \textnormal{deg}(a), \textnormal{deg}(b)\rangle }    ba )\\
 \nonumber & = \big (a\otimes 1 + 1 \otimes a   \big)\big( b \otimes 1 + 1 \otimes b\big) - (-1)^{\langle \textnormal{deg}(a) , \textnormal{deg}(b)\rangle }     \big( b \otimes 1 + 1 \otimes b\big)(a\otimes 1 + 1 \otimes a   \big)\\
 \nonumber & = ab \otimes 1 + a \otimes b + (-1)^{\langle \textnormal{deg}(a), \textnormal{deg}(b)\rangle } b \otimes a + 1 \otimes ab\\
 \nonumber & - (-1)^{\langle \textnormal{deg}(a), \textnormal{deg}(b)\rangle } \big( ba \otimes 1 + b \otimes a + (-1)^{\langle \textnormal{deg}(a), \textnormal{deg}(b)\rangle }    a \otimes b + 1 \otimes ba  \big)\\
&= [a,b]\otimes 1 + 1\otimes [a,b].
\end{align}
Thus, the set of primitive elements is closed under the $\mathbb{Z}_2^2$-graded commutator.
\end{proof}

\begin{definition}{Definition}\label{def:HopfAlg}
A \emph{Hopf $\mathbb{Z}_2^2$-algebra} is a $\mathbb{Z}_2^2$-bialgebra admitting an antipode, that is a  $\mathbb{Z}_2^2$-algebra antihomomorphism  $\mathcal{S} : \mathcal{A} \rightarrow \mathcal{A}$, such that $\mathcal{S}(ab) = (-1)^{\langle \textnormal{deg}(a), \textnormal{deg}(b) \rangle} \mathcal{S}(b)\mathcal{S}(a)$, that satisfies
$$\mu \circ (\mathcal{S} \otimes \textnormal{Id}_\mathcal{A})\circ \Delta = \mu \circ (\textnormal{Id}_\mathcal{A} \otimes \mathcal{S}) \circ \Delta = \eta \circ \epsilon   .$$
A Hopf $\mathbb{Z}_2^2$-algebra is thus a tuple $(\mathcal{A}, \mu, \eta, \Delta, \epsilon, \mathcal{S})$.
\end{definition}
In practice and  where no confusion can arise, we will denote a  Hopf $\mathbb{Z}_2^2$-algebra  $(\mathcal{A}, \mu, \eta, \Delta, \epsilon, \mathcal{S})$ simply as $\mathcal{A}$, understanding all required structure maps as being implied.

Let us denote the interchange map as  $\sigma : \mathcal{A} \otimes \mathcal{A} \rightarrow \mathcal{A} \otimes \mathcal{A} $, which is defined as $\sigma(a \otimes b) = (-1)^{\langle \textnormal{deg}(a), \textnormal{deg}(b)\rangle }    b \otimes a$.
\begin{definition}{Definition}\label{def:commcocomm}
A Hopf $\mathbb{Z}_2^2$-algebra  $\mathcal{A}$ is said to be \emph{commutative} if it is $\mathbb{Z}_2^2$-commutative as an algebra, i.e., $\mu\circ \sigma = \mu$. Similarly, a Hopf $\mathbb{Z}_2^2$-algebra is said to be \emph{cocommutative} if it is $\mathbb{Z}_2^2$-cocommutative as  a coalgebra, i.e. $\sigma \circ \Delta = \Delta$.
\end{definition}
\begin{definition}{Definition}\label{def:involutive}
A Hopf $\mathbb{Z}_2^2$-algebra $\mathcal{A}$ is said to be \emph{involutive} if the antipode satisfies $\mathcal{S}^{2} =  \textnormal{Id}_{\mathcal{A}}$
\end{definition}

\subsection{Bicovariant differential calculus}
 In the following, we will use  the canonical embedding $\mathbb{Z}_2^2 \hookrightarrow \mathbb{N} \times \mathbb{Z}_2^2$ given by $(\gamma_1, \gamma_2) \mapsto (0, \gamma_1, \gamma_2)$. Note that a $\mathbb{N} \times \mathbb{Z}_2^2$-grading descends to a natural  $\mathbb{Z}_2^3$-grading. We will use this fact and employ the $\mathbb{Z}_2^3$-graded tensor product.
\begin{definition}{Definition}\label{def:DiffCal}
Let $\mathcal{A}$ be a Hopf $\mathbb{Z}_2^2$-algebra and let $\Omega^p(\mathcal{A})$ be the $\mathcal{A}$-bimodule of $p$-forms. A \emph{higher-order differential calculus} on $\mathcal{A}$ is the $\mathbb{N} \times \mathbb{Z}_2^2$-graded algebra $\Omega(\mathcal{A}) = \oplus_{p = 0}^{\infty} \Omega^p(\mathcal{A})$ such that $\Omega_{(0, *)}(\mathcal{A}) = \Omega^{0}(\mathcal{A}) \cong \mathcal{A}$, and $\Omega_{(p,*)}(\mathcal{A}) = \Omega^p(\mathcal{A})$,  together with a linear map, the \emph{de Rham differential}, $\mathrm{d}: \Omega^p(\mathcal{A}) \rightarrow\Omega^{p+1}(\mathcal{A})$ of $\mathbb{N} \times \mathbb{Z}_2^2$-degree $(1,0,0)$ that satisfies
\begin{enumerate}
\item $\mathrm{d}^2 =0$,
\item $\mathrm{d}(\alpha \beta) = (\mathrm{d}\alpha) \beta + (-1)^p\alpha \mathrm{d} \beta$,\\ where $\alpha \in \Omega^p(\mathcal{A})$ and $\beta \in \Omega(\mathcal{A})$, and,
\item $\Omega(\mathcal{A})$ is generated by $\mathcal{A}$ and $\Omega^1(\mathcal{A}) := \textrm{Span}\big \{ a     \mathrm{d}b \big\}$, where $a$ and $b \in \mathcal{A}$.
\end{enumerate}
\end{definition}
\begin{remark}{Remark}
The notion of a differential calculi on a quantum group can be traced back to Woronowicz \cite{woro89}. The above definition with regards to the grading is very similar to the conventions of Deligne \cite{del/fre99} for differential forms on supermanifolds  which naturally come with a $\mathbb{N} \times \mathbb{Z}_2$ grading, but form a $\mathbb{Z}_2 \times \mathbb{Z}_2$-commutative algebra. We also note for some generalisations:  graded differential algebras with $d^N=0$ have been developed by  Kapranov \cite{kap}, Dubois-Violette \cite{dub},   Abramov \cite{Abr}, and  Abramov and Kerner \cite{Abr/Ker}; for bicovariant differential and codifferential calculi on finite groups, see \cite{dup2018c}; a  q-deformed  differential calculus  on  the light-cone was given by Akulov,  Duplij and Chitov \cite{aku/dup/chi}, and it allowed the construction of q-twistors and so a q-deformed differential calculi of q-tensors of any rank.
\end{remark}
\begin{definition}{Definition}\label{def:lrcoaction}
Let $\mathcal{A}$ be a Hopf $\mathbb{Z}_2^2$-algebra and let $(\Omega(\mathcal{A}), \mathrm{d})$ be a differential calculus over $\mathcal{A}$. Then $(\Omega(\mathcal{A}), \mathrm{d})$ is said to be
\renewcommand{\theenumi}{\roman{enumi}}
\begin{enumerate}
\item \emph{left-covariant} if there exists a linear map $\Delta_L : \Omega(\mathcal{A}) \rightarrow \mathcal{A}\otimes \Omega(\mathcal{A})$, called a \emph{left coaction}, such that
$$\Delta_L(a \mathrm{d}b) = \Delta(a) (\textnormal{Id}_\mathcal{A} \otimes \mathrm{d})\Delta(b),$$
for all $a,b \in \mathcal{A}$.
\item  \emph{right-covariant} if there exists a linear map $\Delta_R : \Omega(\mathcal{A}) \rightarrow  \Omega(\mathcal{A})\otimes \mathcal{A}$, called a \emph{right coaction}, such that
$$\Delta_R(a \mathrm{d}b) = \Delta(a) ( \mathrm{d} \otimes \textnormal{Id}_\mathcal{A})\Delta(b),$$
for all $a,b \in \mathcal{A}$.
\end{enumerate}
Furthermore, a left-covariant and right-covariant differential calculus $(\Omega(\mathcal{A}), \mathrm{d})$ is said to be \emph{bicovariant}.
\end{definition}
\begin{remark}{Remark}
In the current context, we have a similar result to Woronowicz  \cite[Proposition 1.4]{woro89}. In particular, if $(\Omega(\mathcal{A}), \mathrm{d})$ is a bicovariant differential calculus over $\mathcal{A}$. Then
$$\big(\Delta_L \otimes \textnormal{Id}_\mathcal{A} \big)\circ \Delta_R =  \big(\textnormal{Id}_\mathcal{A}\otimes \Delta_R)\circ \Delta_L   .$$
We will not need this result in this paper and so omit the proof.
\end{remark}
The bicovariance  can be restated as the following conditions:
\begin{align}
 \nonumber & \Delta_L(a     \mathrm{d} a + \mathrm{d} b   b ) = \Delta(a)\Delta_L(\mathrm{d}a) + \Delta_L(\mathrm{d}b)\Delta(b),\\
 & \Delta_R(a     \mathrm{d} a + \mathrm{d} b   b ) = \Delta(a)\Delta_R(\mathrm{d}a) + \Delta_R(\mathrm{d}b)\Delta(b).
\end{align}
\begin{remark}{Remark}\label{rem:bicovariance}
It is clear that we do not actually need a Hopf algebra structure to define left-covariance or right-covariance, but rather just the structure of a bialgebra. That is, the antipode plays no r\^{o}le here.
\end{remark}

\section{The double-graded quantum superplane $\mathbb{R}_q(1|\mathbf{1})$}\label{sec:QuantDoubSupPlane}
\subsection{The double-graded quantum superplane}
Consider the algebra of polynomials with $\mathbb{Z}_2^2$-graded generators
\begin{equation}
\big( \underbrace{x}_{(0,0)},~ \underbrace{\xi}_{(0,1)},~ \underbrace{\theta}_{(1,0)},~ \underbrace{z}_{(1,1)}\big),
\end{equation}
subject to the relations:
\begin{subequations}
\begin{align*}
 & x \xi - q     \xi x = 0, \Label{Eqn:comrel1}
&& x \theta - q     \theta x = 0, \Label{Eqn:comrel2}\\
 & x z -   z x =0,  \Label{Eqn:comrel3}
&& \xi^2 =0,  \Label{Eqn:comrel4}\\
& \theta^2 = 0, \Label{Eqn:comrel5}
&& \xi \theta - \theta \xi =0, \Label{Eqn:comrel6}\\
& \xi z + q^{-1}     z \xi = 0, \Label{Eqn:comrel7}
&& \theta z + q^{-1}     z \theta = 0, \Label{Eqn:comrel8}
\end{align*}
\end{subequations}
where $q \in \mathbb{C}_*$ and is \emph{not} a root of unity.  Note that setting  $q =1$ reduces the relations to $\mathbb{Z}_2^2$-commutativity (see for example \cite{Cov/Gra/Pon1}).
\begin{definition}{Definition}
The $\mathbb{Z}_2^2$-graded, associative, unital algebra
$$\mathcal{A}_q(x, \xi, \theta, z) := \mathbb{C}[x,\xi, \theta, z] / \mathbf{J},$$
where $\mathbf{J}$ is the ideal generated by the relations \eqref{Eqn:comrel1} to \eqref{Eqn:comrel8}  is the \emph{algebra of polynomials on the   double-graded quantum superplane} $\mathbb{R}_q(1|\mathbf{1})$.
\end{definition}
The relations \eqref{Eqn:comrel1} to \eqref{Eqn:comrel8} should, of course, be compared  with the relations that define Manin's superplane $\mathbb{R}_q(1|2)$  (see \cite{man89}).  Manin considers the generators $\{x', \xi', \theta'   \}$ of $\mathbb{Z}_2$-degree $0$, $1$ and $1$, respectively, all subject to the following relations:
\begin{subequations}
\begin{align}
 & x' \xi' - q     \xi' x' = 0, \label{Eqn:Mancomrel1} \\
& x' \theta' - q     \theta' x' = 0, \label{Eqn:Mancomrel2}\\
& \xi' \theta' + q^{-1}    \theta' \xi' =0, \label{Eqn:Mancomrel3}\\
& \xi'^2 =0,  \label{Eqn:Mancomrel4}\\
& \theta'^2 = 0. \label{Eqn:Mancomrel5}
\end{align}
\end{subequations}
 In particular, notice that \eqref{Eqn:comrel4} and \eqref{Eqn:comrel5}  show that $\xi$ and $\theta$ are nilpotent, but \eqref{Eqn:comrel6} means that they mutually commute rather than anticommute -- this is, neglecting the factor of $q^{-1}$, the opposite of \eqref{Eqn:Mancomrel3}. That is, they are `self-fermions' but are `relative-bosons'. Moreover, $z$ is not nilpotent, however, it satisfies a twisted anticommutation relation with both $\xi$ and $\theta$, see \eqref{Eqn:comrel7} and \eqref{Eqn:comrel8}, and compare with \eqref{Eqn:Mancomrel3}. Thus, $z$ is a `self-boson' but is a `relative-fermion' with respect to $\xi$ and $\theta$.  The language here is borrowed from the theory of Green--Volkov parastatistics (see \cite{Gre,Vol}). Many of the following constructions and mathematical results will closely parallel than of Manin's superplane, but with subtle sign differences due to the novel $\mathbb{Z}_2^2$-grading we employ.\par
For brevity, we will set $\mathcal{A}_q := \mathcal{A}_q(x, \xi, \theta, z)$.  Let $\mathcal{A}_{q,k}$  $(k \in \mathbb{N})$  be the homogeneous component of $\mathcal{A}_q$ spanned by monomials of the form
\begin{equation}\label{Eqn:PBWbasis}
x^m \xi^{\alpha} \theta^{\beta} z^n   ,
\end{equation}
i.e., we use a PBW-like basis, where $m+\alpha + \beta + n =k$. Note that $m,n \in \mathbb{N}$, while due to the nilpotent nature of $\xi$ and $\theta$, $\alpha, \beta \in \{0,1\}$.

\subsection{The $\mathbb{Z}_2^2$-bialgebra structure on the double-graded quantum superplane}
As well as a $\mathbb{Z}_2^2$-algebra structure, we naturally have a  $\mathbb{Z}_2^2$-bialgebra structure on the double-graded quantum superplane.
\begin{proposition}{Proposition}\label{prop:bialgq}
The following coproduct and counit provide $\mathcal{A}_q(x, \xi, \theta, z)$ with the structure of a $\mathbb{Z}_2^2$-bialgebra (see Definition \ref{def:coalgebra} and Proposition \ref{prop:Z22counit}):
\begin{align}
& \Delta(x) = x \otimes x, \\
& \Delta(\xi) = x \otimes \xi + \xi \otimes x,\\
& \Delta(\theta) = x \otimes \theta + \theta \otimes x , \\
& \Delta(z) = x \otimes z + z \otimes x, \\
& \epsilon(x) = 1, \\
& \epsilon(\xi) = \epsilon(\theta) = \epsilon(z) = 0.
\end{align}
\end{proposition}
\begin{proof}
We need to show that the above defined coproduct and counit do indeed define a $\mathbb{Z}_2^2$-coalgabra (see Definition \ref{def:coalgebra}). It is sufficient to check these conditions on each generator separately.  First, we check the counit condition:
\renewcommand{\theenumi}{\roman{enumi}}
\begin{enumerate}
 \setlength\itemsep{1em}
\item $(\epsilon \otimes \textnormal{Id})\Delta(x)= \epsilon(x)\otimes x = 1 \otimes x \simeq x$,\\
$(\textnormal{Id}\otimes \epsilon)\Delta(x)=  x \otimes \epsilon(x) =  x  \otimes 1 \simeq x$.
\item $(\epsilon \otimes \textnormal{Id})\Delta(\xi)= \epsilon(x)\otimes \xi = 1 \otimes \xi \simeq \xi$,\\
$(\textnormal{Id}\otimes \epsilon)\Delta(\xi)=  \xi \otimes \epsilon(x) =  \xi  \otimes 1 \simeq \xi$.
\item The same calculation as in part (ii) shows that the counit condition holds for $\theta$ and $z$.
\end{enumerate}
Secondly, we check the coassociativity:
\begin{enumerate}
\setlength\itemsep{1em}
\setcounter{enumi}{3}
\item $(\Delta \otimes \textnormal{Id})\Delta(x)=  x\otimes x \otimes x$,\\
$(\textnormal{Id}\otimes \Delta)\Delta(x)=  x \otimes x \otimes x$.
\item $(\Delta \otimes \textnormal{Id})\Delta(\xi)= (\Delta \otimes \textnormal{Id})(x\otimes \xi + \xi \otimes x) = x \otimes x \otimes \xi +  x \otimes \xi \otimes x   +\xi \otimes x \otimes x $,\\
$(\textnormal{Id}\otimes \Delta)\Delta(\xi)= (\textnormal{Id}\otimes \Delta)(x\otimes \xi + \xi \otimes x) =  x \otimes x \otimes \xi  +  x \otimes \xi \otimes x + \xi \otimes x \otimes x $.
\item The same calculation as in part (v) shows that the coassociativity condition holds for $\theta$ and $z$.
\end{enumerate}
 Note that we have a cocommutative $\mathbb{Z}_2^2$-coalgabra (see Definition \ref{def:commcocomm}).\par
Thirdly, we check that the algebra and coalgebra structure are compatible. We do this by showing that the coproduct is an algebra morphism (see Definition \ref{def:bialgebra}). This requires direct calculations:
\begin{enumerate}
\setlength\itemsep{1em}
\setcounter{enumi}{6}
\item \begin{align*} \Delta(x)\Delta(\xi)  &=   (x\otimes x)(x \otimes \xi + \xi \otimes x) \\
        &= x^2  \otimes x \xi + x \xi \otimes x^2\\
         &=q ( x^2 \otimes \xi x + \xi x \otimes  x^2)\\
        &= q (x\otimes \xi + \xi \otimes x )(x \otimes               x)\\
         &= q \Delta(\xi)\Delta(x).
    \end{align*}
\item An  identical calculation to part (vii) upon replacing $\xi$ with $\theta$ shows that
\begin{equation*}
\Delta(x)\Delta(\theta) = q \Delta(\theta)\Delta(x).
\end{equation*}
\item \begin{align*}
      \Delta(x)\Delta(z) & =   (x\otimes x)(x \otimes z +z \otimes x) \\
       & = x^2  \otimes x z + x z \otimes x^2\\
         & = x^2 \otimes z x + z x \otimes  x^2\\
        & =  (x\otimes z + z \otimes x )(x \otimes                  x)\\
        & = \Delta(z)\Delta(x).
       \end{align*}
\item  $\Delta(\xi)\Delta(\xi)=0$ is  obviously satisfied. Direct calculation show this to be consistent.
\begin{align*}
 \Delta(\xi) \Delta(\xi) &= (x \otimes \xi + \xi \otimes x)(x \otimes \xi + \xi \otimes x)\\
 \nonumber & = - x \xi \otimes \xi x + \xi x \otimes x \xi\\
 & = - q q^{-1}\xi x \otimes x \xi + \xi x \otimes x \xi =0.
\end{align*}
\item $\Delta(\theta)\Delta(\theta)=0$ follows in the same way as in part (x).
\item  \begin{align*}
      \Delta(\xi)\Delta(\theta) & =  (x \otimes \xi + \xi \otimes x)(x\otimes \theta + \theta \otimes x)\\
     & = x^2 \otimes \xi \theta + x \theta \otimes \xi x + \xi x \otimes x \theta + \xi \theta \otimes x^2 \\
      & = x^2 \otimes \theta \xi + \theta x \otimes x \xi + x \xi \otimes \theta x + \theta \xi \otimes x^2\\
       & =  (x\otimes \theta + \theta \otimes x)(x \otimes \xi + \xi \otimes x )\\
       & = \Delta(\theta)\Delta(\xi).
       \end{align*}
\item \begin{align*}
\Delta(\xi)\Delta(z)  & = (x \otimes \xi + \xi \otimes x )(x \otimes z + z\otimes x)\\
 & = x^2 \otimes \xi z - x z \otimes \xi x + \xi x\otimes x z + \xi z \otimes x^2\\
  & = - q^{-1}(x^2 \otimes z \xi + z x \otimes x \xi - x \xi \otimes z x + z \xi \otimes x^2)\\
  &= - q^{-1}(x\otimes z + z \otimes x)(x \otimes \xi + \xi \otimes x)\\
  &= - q^{-1}\Delta(z)\Delta(\xi).
\end{align*}
\item The identical calculation to part (viii)  upon replacing $\xi$ with $\theta$  shows that
\begin{equation}
\Delta(\theta)\Delta(z) = - q^{-1}\Delta(z)\Delta(\theta).
\end{equation}
\end{enumerate}
This completes the proof.
\end{proof}
\subsection{The extended  double-graded quantum superplane and its Hopf algebra}
In order to build a $\mathbb{Z}_2^2$-Hopf algebra structure (see Definition \ref{def:HopfAlg}), we now extend the algebra of polynomials on the  double-graded quantum superplane to include the formal inverse of $x$, which we denote as $x^{-1}$, i.e., $x x^{-1} = x^{-1}x =1$. Clearly, the $\mathbb{Z}_2^2$-degree of $x^{-1}$ is $(0,0)$. It is easy to deduce the following commutation rules:
\begin{subequations}
\begin{align}
&x^{-1} \xi - q^{-1}\xi x^{-1} =0, \label{eqn:invcomrelq1} \\
& x^{-1} \theta - q^{-1}\theta x^{-1} =0 , \label{eqn:invcomrelq2}\\
 & x^{-1} z - z x^{-1} =0.\label{eqn:invcomrelq3}
\end{align}
\end{subequations}
\begin{definition}{Definition}{Definition}\label{def:extdoubleqplane}
The $\mathbb{Z}_2^2$-graded, associative, unital algebra
$$\mathcal{A}_q(x, x^{-1}, \xi, \theta, z) := \mathbb{C}[x,x^{-1}, \xi, \theta, z] / \mathbf{J},$$
where $\mathbf{J}$ is the ideal generated by the relations \eqref{Eqn:comrel1} to \eqref{Eqn:comrel8} and \eqref{eqn:invcomrelq1} to \eqref{eqn:invcomrelq3} is the \emph{algebra of polynomials on the extended  double-graded quantum superplane} $\overline{\mathbb{R}}_q(1|\mathbf{1})$.
\end{definition}

As the coproduct should be an algebra morphism we deduce that $\Delta(x^{-1}) = \Delta(x)^{-1}$. Thus,
\begin{equation}
 \Delta(x^{-1})= x^{-1}\otimes x^{-1}.
 \end{equation}
 Similarly, as the counit should be an algebra morphism we have
 that
 \begin{equation}
 \epsilon(x^{-1}) =1.
 \end{equation}
It is clear that upon appending $x^{-1}$ to the $\mathbb{Z}_2^2$-bialgebra  $\mathcal{A}_q(x, \xi, \theta, z)$ that we obtain another $\mathbb{Z}_2^2$-bialgebra. The counit and coassociativity are obvious and the compatibility condition between the algebra and coalgebra follows from the proof of Proposition \ref{prop:bialgq} upon $x \mapsto x^{-1}$ and $q \mapsto q^{-1}$.
\begin{theorem}{Theorem}
The $\mathbb{Z}_2^2$-bialgebra  $\mathcal{A}_q(x, x^{-1}, \xi, \theta, z)$ can be made into a (cocommutative and involutive) $\mathbb{Z}_2^2$-Hopf algebra by defining an antipode in the following way:
\begin{subequations}
\begin{align}
& \mathcal{S}(x) = x^{-1}, \\
& \mathcal{S}(x^{-1}) = x,\\
& \mathcal{S}(\xi) = - x^{-1} \xi x^{-1},\\
& \mathcal{S}(\theta) = - x^{-1} \theta x^{-1},\\
& \mathcal{S}(z) = - x^{-1} z x^{-1}.
\end{align}
\end{subequations}
\end{theorem}{Theorem}
\begin{proof}\
We need to check that the antipode satisfies the condition specified in Definition \ref{def:HopfAlg}. It suffices to check this on each generator separately.
\renewcommand{\theenumi}{\roman{enumi}}
\begin{enumerate}
\setlength\itemsep{1em}
\item  $\mu(\mathcal{S}\otimes \textnormal{Id})\Delta(x) = \mu(x^{-1} \otimes x) =1 = \mu(x \otimes  x^{-1}) = \mu(\textnormal{Id} \otimes \mathcal{S})\Delta(x)$.
\item $\mu(\mathcal{S} \otimes \textnormal{Id})\Delta(\xi)  =  \mu(\mathcal{S}(x) \otimes \xi + \mathcal{S}(\xi) \otimes x)
 = \mu( x^{-1}\otimes \xi  - x^{-1} \xi x^{-1}\otimes x) =0. $\\
$ \mu( \textnormal{Id}\otimes \mathcal{S})\Delta(\xi) = \mu(x \otimes \mathcal{S}(\xi)  + \xi \otimes \mathcal{S}(x))
= \mu( x \otimes (- x^{-1}\xi x^{-1}) + \xi \otimes x^{-1} ) =0$.
\item An identical calculation to part (ii) shows that the required condition also holds for $\theta$ and $z$.
\end{enumerate}
It is clear that the coproduct is cocommutative and a simple calculation shows that $\mathcal{S}^2 = \textnormal{Id}$ (see Definition \ref{def:commcocomm} and Definition \ref{def:involutive}).
\end{proof}
\subsection{A bicovariant differential calculus} We build the differential calculus (see Definition \ref{def:DiffCal}) on the double-graded quantum superplane $\mathbb{R}_{q}(1|\mathbf{1})$ (see Remark \ref{rem:bicovariance}) using the following basis of one-forms
\begin{equation}\label{eqn:basisoneform}
(\underbrace{\mathrm{d}x}_{(1,0,0)}, \underbrace{\mathrm{d} \xi}_{(1,0,1)},\underbrace{\mathrm{d} \theta}_{(1,1,0)} , \underbrace{\mathrm{d}z}_{(1,1,1)})
\end{equation}
Using Definition \ref{def:lrcoaction}, the left coaction and right coaction of the basis of one-forms given in \eqref{eqn:basisoneform} is given by
\begin{subequations}
\begin{align}
& \Delta_L(\mathrm{d}x) = x \otimes \mathrm{d}x,\\
& \Delta_L(\mathrm{d}\xi) = x \otimes \mathrm{d}\xi + \xi \otimes \mathrm{d}x,\\
& \Delta_L(\mathrm{d}\theta) = x \otimes \mathrm{d}\theta + \theta \otimes \mathrm{d}x,\\
& \Delta_L(\mathrm{d}z) = x \otimes \mathrm{d}z + z \otimes \mathrm{d}x,
\end{align}
\end{subequations}
and
\begin{subequations}
\begin{align}
& \Delta_R(\mathrm{d}x) =\mathrm{d}x \otimes x,\\
& \Delta_R(\mathrm{d}\xi) = \mathrm{d}x \otimes \xi + \mathrm{d}\xi \otimes x,\\
& \Delta_R(\mathrm{d}\theta) = \mathrm{d}x \otimes \theta +\mathrm{d} \theta \otimes x,\\
& \Delta_R(\mathrm{d}z) = \mathrm{d} x \otimes z + \mathrm{d}z \otimes x.
\end{align}
\end{subequations}
We now proceed to find a consistent set of commutation rules on this differential calculi. To be explicit, by taking the de Rham derivative of the commutation rules for $(x, \xi, \theta, z) $ we arrive at the following.
\begin{lemma}{Lemma}\label{lem:condiff}
Any commutation rules between the coordinates/generators and their differentials must satisfy the following:
\begin{subequations}
\begin{align}
\setlength\itemsep{0em}
& (x    \mathrm{d}\xi - q     \mathrm{d}\xi    x) - q(\xi     \mathrm{d}x  - q^{-1}    \mathrm{d}x     \xi) =0,\\
& (x    \mathrm{d}\theta - q     \mathrm{d}\theta    x) - q(\theta     \mathrm{d}x  - q^{-1}    \mathrm{d}x     \theta) =0,\\
& (x    \mathrm{d}z-      \mathrm{d}z    x) - (z     \mathrm{d}x  -     \mathrm{d}x     z) =0,\\
&(\xi    \mathrm{d}\theta -      \mathrm{d}\theta    \xi) - (\theta     \mathrm{d}\xi  -     \mathrm{d}\xi    \theta) =0,\\
& (\xi    \mathrm{d}z + q^{-1}     \mathrm{d}z    \xi) + q^{-1}(z     \mathrm{d}\xi  + q    \mathrm{d}\xi     z) =0,\\
& (\theta    \mathrm{d}z + q^{-1}     \mathrm{d}z    \theta) + q^{-1}(z     \mathrm{d}\theta  + q    \mathrm{d}\theta    z) =0.
\end{align}
\end{subequations}
\end{lemma}
We propose a particular set of commutation rules between the coordinates and the differentials that is linear in the coordinates and, in particular, treats the two sectors $(\xi, \mathrm{d}\xi)$ and $(\theta, \mathrm{d} \theta)$ equally. The reader should compare our choice the Type I differential calculi on the quantum superplane $\mathbb{R}_q(1|1)$ as given by \c{C}elik \cite{cel06}.
\begin{theorem}{Theorem}\label{thm:valcomrules}
A set of valid  commutation rules (in the sense of Lemma \ref{lem:condiff} and is consistent with the bicovariance) that is linear in $(x, \xi, \theta, z) $ is the following:
\begin{subequations}
\begin{align*}
& x   \mathrm{d}x =  \mathrm{d}x    x, \Label{Eqn:ValCom1}&&  x     \mathrm{d}\xi = q     \mathrm{d}\xi     x , \Label{Eqn:ValCom2}\\
&  x     \mathrm{d}\theta = q     \mathrm{d}\theta    x,\Label{Eqn:ValCom3}&& x     \mathrm{d}z =  \mathrm{d}z    x,\Label{Eqn:ValCom4} \\
\\
& \xi     \mathrm{d}x =  q^{-1}    \mathrm{d}x     \xi,  \Label{Eqn:ValCom5} &&\xi     \mathrm{d}\xi = - \mathrm{d} \xi     \xi,\Label{Eqn:ValCom6}\\
& \xi     \mathrm{d}\theta =  \mathrm{d}\theta     \xi, \Label{Eqn:ValCom7} && \xi     \mathrm{d}z = -q^{-1}     \mathrm{d}z     \xi, \Label{Eqn:ValCom8}\\
\\
& \theta     \mathrm{d}x =   q^{-1}    \mathrm{d}x     \theta,\Label{Eqn:ValCom9} &&  \theta    \mathrm{d}\xi = \mathrm{d}\xi    \theta, \Label{Eqn:ValCom10}\\
& \theta     \mathrm{d}\theta= - \mathrm{d} \theta     \theta, \Label{Eqn:ValCom11} && \theta     \mathrm{d}z = -q^{-1}     \mathrm{d}z     \theta, \Label{Eqn:ValCom12}\\
\\
&z     \mathrm{d}x = \mathrm{d}x     z, \Label{Eqn:ValCom13}&& z    \mathrm{d}\xi = - q    \mathrm{d}\xi    z ,\Label{Eqn:ValCom14}\\
& z \mathrm{d}\theta =  - q \mathrm{d}\theta    z , \Label{Eqn:ValCom15} && z    \mathrm{d}z =  \mathrm{d}z    z \Label{Eqn:ValCom16}.
\end{align*}
\end{subequations}
\end{theorem}
\begin{proof}
It is a simple observation that the above relations satisfy the conditions of Lemma \ref{lem:condiff}. It remains to check that these relations are consistent with the bicovariance (see Definition \ref{def:lrcoaction}). This is a series of direct computations.  For instance, consider the commutation rule \eqref{Eqn:ValCom2}.
\begin{align*}
\Delta_L(x     \mathrm{d}\xi - q     \mathrm{d}\xi     x) & =  \Delta(x) \Delta_L(\mathrm{d}\xi) - q     \Delta_L(\mathrm{d}) \Delta(x)\\
& = (x\otimes x)(x \otimes \mathrm{d}\xi + \xi \otimes \mathrm{d}x) - q     (x \otimes \mathrm{d}\xi + \xi \otimes \mathrm{d}x)(x\otimes x)\\
&= x^2 \otimes x \mathrm{d}\xi  + x \xi \otimes x \mathrm{d}x - q    x^2 \otimes \mathrm{d}\xi x - q    \xi x \otimes \mathrm{d}x x\\
&=0.
\end{align*}
Thus, \eqref{Eqn:ValCom2} respects the left-covariance.
\begin{align*}
\Delta_R(x     \mathrm{d}\xi - q     \mathrm{d}\xi     x) & =  \Delta(x) \Delta_R(\mathrm{d}\xi) - q     \Delta_R(\mathrm{d}) \Delta(x)\\
& = (x\otimes x)(\mathrm{d}x \otimes \xi + \mathrm{d}\xi \otimes x) - q     (\mathrm{d}x \otimes \xi + \mathrm{d}\xi \otimes x)(x\otimes x)\\
&= x \mathrm{d}x \otimes x \xi + x \mathrm{d}\xi \otimes x^2 -  q    \mathrm{d}xx \otimes \xi x\ - q    \mathrm{d}\xi x \otimes x^2\\
&=0.
\end{align*}
And so we observe that  \eqref{Eqn:ValCom2} also respects the right-covariance and so bicovariance is established. All the other commutation relations can be shown to respect the bicovariance via almost identical calculation and so we omit details.
\end{proof}
\begin{remark}{Remark}
There are clearly other choices of differential calculi that could be made. The classification of the possible bicovariant differential calculi is an important question. However, we will not touch on this in this paper.
\end{remark}
In order to construct higher order differential forms we need to deduce the commutation rules between the differentials. This is easily achieved by applying the de Rham differential to Theorem \ref{thm:valcomrules}.
\begin{theorem}{Theorem}\label{thm:ComRulesxdx}
The (non-trivial) commutation rules between the differentials are:
\begin{subequations}
\begin{align}
&\mathrm{d}x     \mathrm{d}\xi = - q    \mathrm{d}\xi     \mathrm{d}x,\\
&\mathrm{d}x     \mathrm{d}\theta = - q     \mathrm{d}\theta     \mathrm{d}x,\\
&\mathrm{d}x     \mathrm{d}z = - \mathrm{d}z     \mathrm{d}x,\\
\nonumber \\
&\mathrm{d}\xi     \mathrm{d}\theta = -  \mathrm{d}\theta     \mathrm{d}\xi,\\
&\mathrm{d}\xi     \mathrm{d}z =  q^{-1}   \mathrm{d}z     \mathrm{d}\xi,\\
\nonumber \\
&\mathrm{d}\theta     \mathrm{d}z =  q^{-1}    \mathrm{d}z    \mathrm{d}\theta.
\end{align}
\end{subequations}
Moreover,  $(\mathrm{d}x)^2 = (\mathrm{d}z)^2 =0$.
\end{theorem}
\begin{proof}
Direct computation gives the mixed commutation rules and so we omit details. The nilpotency of $\mathrm{d}x$ and $\mathrm{d}z$ directly follows as, for example, $\mathrm{d}(x    \mathrm{d}x) = \mathrm{d}x \mathrm{d}x $, but then using the fact that $x$ and $\mathrm{d}x$ strictly commute    $\mathrm{d}(x    \mathrm{d}x) = \mathrm{d}( \mathrm{d}x   x) = - \mathrm{d}x    \mathrm{d}x $. Exactly the same reasoning establishes that $\mathrm{d}z$ is also nilpotent.
\end{proof}
Theorem \ref{thm:valcomrules} and Theorem \ref{thm:ComRulesxdx} allow one to deduce the explicit form of a  bicovariant differential calculi  (see Definition \ref{def:DiffCal}) on $\mathbb{R}_q(1|\mathbf{1})$.
\begin{remark}{Remark}
Unsurprisingly, just as for supermanifolds and $\mathbb{Z}_2^n$-manifolds, there are \emph{no} top-forms on $\mathbb{R}_q(1|\mathbf{1})$ due to the fact that $\mathrm{d}\xi$ and $\mathrm{d}\theta$ are \emph{not} nilpotent. To see this one has to observe that, for instance, $\xi$ and $\mathrm{d}\xi$ strictly anticommute. This extra minus sign does not allow us to conclude that  $\mathrm{d}\xi$ is nilpotent.
\end{remark}
We now deduce the commutation relations between the partial derivatives $\{\partial_x, \partial_{\xi}, \partial_{\theta}, \partial_z \}$ and the generators/coordinates on the double-graded quantum superplane.  This is done by careful examination of the de Rham differential, which is of the form
\begin{equation}
\mathrm{d} = \mathrm{d}x \partial_x + \mathrm{d}\xi \partial_{\xi} + \mathrm{d}\theta \partial_{\theta} + \mathrm{d}z \partial_z.
\end{equation}
\begin{theorem}{Theorem}\label{thm:ComPartCoor}
The commutation rules between partial derivatives and the coordinates are:
\begin{subequations}
\begin{align*}
& \partial_x x  = 1 + x \partial_x, \Label{eqn:ComPartCoor1}
&& \partial_{\xi} x  = q     x \partial_\xi, \Label{eqn:ComPartCoor2}\\
& \partial_{\theta} x = q     x \partial_{\theta}, \Label{eqn:ComPartCoor3} && \partial_z x  =  x \partial_z, \Label{eqn:ComPartCoor4}\\
 \\
& \partial_x \xi  = q^{-1}     \xi \partial _x, \Label{eqn:ComPartCoor5}
&& \partial_{\xi} \xi  = 1 - \xi \partial_{\xi},\Label{eqn:ComPartCoor6} \\
& \partial_{\theta} \xi = \xi \partial_{\theta} , \Label{eqn:ComPartCoo7}
&& \partial_z \xi  =   -  q^{-1}    \xi \partial_z,\Label{eqn:ComPartCoor8}\\
 \\
& \partial_x \theta  = q^{-1}    \theta \partial_x, \Label{eqn:ComPartCoor9}
&& \partial_{\xi} \theta  = \theta \partial \xi, \Label{eqn:ComPartCoor10}\\
& \partial_{\theta} \theta = 1 - \theta \partial_{\theta}, \Label{eqn:ComPartCoor11}
&& \partial_z \theta  =  - q^{-1}     \theta \partial_z , \Label{eqn:ComPartCoor12}\\
 \\
& \partial_x z  =  z \partial_x, \Label{eqn:ComPartCoor13}
&& \partial_{\xi} z  =  - q     z \partial_{\xi}, \Label{eqn:ComPartCoor14}\\
& \partial_{\theta} z =  - q     z \partial_{\theta},\Label{eqn:ComPartCoor15}
&& \partial_z z   = 1 + z \partial_z. \Label{eqn:ComPartCoor16}
\end{align*}
\end{subequations}
\end{theorem}
\begin{proof}
Consider $xf$, where $f \in \mathcal{A}_q$ is arbitrary. Directly from the definition of the de Rham differential, the fact that it satisfies the Leibniz rule and the commutation rules of Theorem \ref{thm:ComRulesxdx} see that
\begin{align*}
\mathrm{d} (xf) & = \mathrm{d}x \partial_x(xf) + \mathrm{d}\xi \partial_{\xi}(xf) + \mathrm{d}\theta \partial_{\theta}(xf) + \mathrm{d}z \partial_z(xf)\\
 & = \mathrm{d}f+ \mathrm{d}x     x\partial_x f + \mathrm{d}\xi     q x \partial_{\xi} f + \mathrm{d}\theta     q x\partial_{\theta} f + \mathrm{d}z      x\partial_z f.
\end{align*}
Equating the terms in the differentials produces the first block of identities, i.e., \eqref{eqn:ComPartCoor1} to \eqref{eqn:ComPartCoor4}. Via an almost identical calculation, by considering $\xi f$ one obtains \eqref{eqn:ComPartCoor5} to \eqref{eqn:ComPartCoor8}. Similarly, by considering $\theta f$  one obtains \eqref{eqn:ComPartCoor9} to \eqref{eqn:ComPartCoor12}  and  $z f$  one obtains \eqref{eqn:ComPartCoor13} to \eqref{eqn:ComPartCoor16}.
\end{proof}
\begin{definition}{Definition}
The $\mathcal{A}_q$-module of \emph{first-order differential operators on the double-graded quantum superplane}, which we denote as $\mathcal{D}^1(\mathcal{A}_q)$, is the $\mathcal{A}_q$-bimodule generated by the partial derivatives $\{\partial_x, \partial_{\xi}, \partial_{\theta}, \partial_z \}$, subject to the relations  \eqref{eqn:ComPartCoor1} to \eqref{eqn:ComPartCoor16}
\end{definition}
\begin{proposition}{Proposition}
The commutation rules  between the partial derivatives are:
\begin{subequations}
\begin{align*}
& \partial_x \partial_{\xi} = q     \partial_{\xi} \partial_x,\Label{Eqn:ComPart1} && \partial_x \partial_{\theta} = q    \partial_{\theta}\partial_x, \Label{Eqn:ComPart2}\\
& \partial_x \partial_z = \partial_z \partial_x, \Label{Eqn:ComPart3}&& \partial_{\xi} \partial_{\theta} = \partial_{\theta} \partial_{\xi}, \Label{Eqn:ComPart4}\\
& \partial_{\xi} \partial_z = - q^{-1}     \partial_z \partial_{\xi}, \Label{Eqn:ComPart5} && \partial_{\theta} \partial_z = - q^{-1}     \partial_z \partial_{\theta}, \Label{Eqn:ComPart6}\\
& \partial_{\xi} \partial_{\xi} =0, \Label{Eqn:ComPart7}&& \partial_{\theta} \partial_{\theta} =0. \Label{Eqn:ComPart8}
\end{align*}
\end{subequations}
\end{proposition}
\begin{proof}
The proof is obtained by comparing the action of the partial derivatives on the PBW basis \eqref{Eqn:PBWbasis}.  For instance,
$$\partial_x \partial_{\xi}(x^m \xi^{\alpha} \theta^{\beta} z^m) = m q^m (x^{m-1}\theta^{\beta} z^n),$$
where we have assumed that $\alpha =1$. On the other hand,
$$\partial_{\xi} \partial_{x}(x^m \xi^{\alpha} \theta^{\beta} z^m) = m q^{m-1} (x^{m-1}\theta^{\beta} z^n).$$
Comparing the two produces \eqref{Eqn:ComPart1}. One obtains \eqref{Eqn:ComPart2} to \eqref{Eqn:ComPart6} in a similar way and so we omit details. The nilpotent nature of $\xi$ and $\theta$ directly imply \eqref{Eqn:ComPart7} and \eqref{Eqn:ComPart8}.
\end{proof}
\begin{proposition}{Proposition}
The commutation rules between the partial derivatives are differentials are:
\begin{subequations}
\begin{align*}
& \partial_x    \mathrm{d}x = \mathrm{d}x    \partial_x , \Label{Eqn:ComParDif1} && \partial_x     \mathrm{d}\xi = q^{-1}     \mathrm{d}\xi     \partial_x , \Label{Eqn:ComParDif2}\\
& \partial_x     \mathrm{d}\theta = q^{-1}     \mathrm{d}\theta     \partial_x , \Label{Eqn:ComParDif3} && \partial_x    \mathrm{d}z = \mathrm{d}z    \partial_x , \Label{Eqn:ComParDif4}\\
\\
& \partial_{\xi}    \mathrm{d}x =  q    \mathrm{d}x    \partial_{\xi}, \Label{Eqn:ComParDif5} &&\partial_{\xi}    \mathrm{d}\xi =  - \mathrm{d}\xi    \partial_{\xi}, \Label{Eqn:ComParDif6}\\
 & \partial_{\xi}    \mathrm{d}\theta =   \mathrm{d}\theta    \partial_{\xi}, \Label{Eqn:ComParDif7} &&\partial_{\xi}    \mathrm{d}z =  -q      \mathrm{d}z    \partial_{\xi}, \Label{Eqn:ComParDif8}\\
 \\
& \partial_{\theta}    \mathrm{d}x =  q    \mathrm{d}x    \partial_{\theta}, \Label{Eqn:ComParDif9} &&\partial_{\theta}    \mathrm{d}\theta =  - \mathrm{d}\theta    \partial_{\theta}, \Label{Eqn:ComParDif10}\\
 & \partial_{\theta}    \mathrm{d}\xi =   \mathrm{d}\xi    \partial_{\theta}, \Label{Eqn:ComParDif11} &&\partial_{\theta}    \mathrm{d}z =  -q      \mathrm{d}z    \partial_{\theta}, \Label{Eqn:ComParDif12}\\
  \\
& \partial_z    \mathrm{d}x =   \mathrm{d}x    \partial_z, \Label{Eqn:ComParDif13} &&\partial_z    \mathrm{d}\xi =  - q^{-1}    \mathrm{d}\xi    \partial_z, \Label{Eqn:ComParDif14}\\
 & \partial_z    \mathrm{d}\theta =  - q^{-1}    \mathrm{d}\theta    \partial_z, \Label{Eqn:ComParDif15} &&\partial_z    \mathrm{d}z =        \mathrm{d}z    \partial_z. \Label{Eqn:ComParDif16}
\end{align*}
\end{subequations}
\end{proposition}
\begin{proof}
The above relations are found by using $\partial_{x^a}(x^b    \mathrm{d}x^c) = \delta_{a}^{\:\: b}     \mathrm{d}x^c$, where we have set $x^a := (x, \xi, \theta,z )$, and applying this to \eqref{Eqn:ValCom1} to \eqref{Eqn:ValCom16}. For instance,
$$\partial_x(x     \mathrm{d}\xi) = \mathrm{d}\xi = q     \partial _x(\mathrm{d}\xi    x),$$
where we have used \eqref{Eqn:ValCom2}. For consistency this implies that
$$\partial_x     \mathrm{d}\xi = q^{-1}     \mathrm{d}\xi     \partial_x,$$
i.e, \eqref{Eqn:ComParDif2} is established. All the other relations follow from similar considerations and so we omit details.
\end{proof}

\section{Concluding remarks}\label{sec:ConRem}
In this paper, we defined a  $\mathbb{Z}_2^2$-graded generalisation of Manin's quantum superplane and presented a bicovariant differential calculi rather explicitly. In this respect, we have a concrete example of a noncommutative differential $\mathbb{Z}_2^2$-geometry. To our knowledge, the double-graded quantum superplane is the first such example to be defined and studied. We must remark that there has been some renewed interest in $\mathbb{Z}_2^n$-gradings in physics, see for example \cite{Aiz/Isa/Seg,Aiz/Kuz/Tan/Top,Aiz/Seg,bru,tol14,tol14a}. It is not clear if these `higher gradings' play a fundamental r\^{o}le in physics in the same way as $\mathbb{Z}_2$-gradings do. However, the results of the aforementioned papers suggest that systems that are $\mathbb{Z}_2^n$-graded are not as uncommon as one might initially think. Thus, we believe, that further work on noncommutative $\mathbb{Z}_2^n$-geometry is warranted and that further links with physics will be uncovered. Indeed, we have only scratched the surface in this paper and have focused on mathematical questions.

\section*{Acknowledgements}
The second author (S.D.) wishes to express his sincere thankfulness and
deep gratitude to Vladimir Akulov, Thomas Nordahl, Vladimir Tkach, Alexander Voronov and Raimund Vogl for numerous fruitful discussions and valuable support.

\mbox{} 


\vskip 2cm
\begin{thebibliography}{10}

\bibitem{Abr}
V.~Abramov:
 Algebra Forms with $d^N = 0$ on Quantum Plane. Generalized
Clifford Algebra Approach, \emph{Advances in Applied Clifford Algebras} \textbf{17},  577--588 (2007).

\bibitem{Abr/Ker}
V.~Abramov, and  R.~Kerner:
Exterior differentials of higher order and their covariant
generalization, \emph{J. Math. Phys.} \textbf{41}, 5598--5614 (2000).

\bibitem{Aiz/Isa/Seg}
N.~Aizawa, P.~S. Isaac, and J.~Segar:
 {$\Bbb Z_2\times\Bbb Z_2$} generalizations of {$\mathcal{N}=1$}
  superconformal {G}alilei algebras and their representations,
 \emph{J. Math. Phys.}
 {\bf 60}, 023507, 11 (2019).

\bibitem{Aiz/Kuz/Tan/Top}
N.~Aizawa, Z.~Kuznetsova, H.~Tanaka, and F.~Toppan:
 $\mathbb{Z}_2\times\mathbb{Z}_2$-graded {L}ie symmetries of the
  {L}\'{e}vy-{L}eblond equations,
 \emph{Prog. Theor. Exp. Phys.}
 {\bf 123A01}, 26 (2016).

\bibitem{Aiz/Seg}
N.~Aizawa and J.~Segar:
 {$\Bbb Z_2\times\Bbb Z_2$} generalizations of {$\mathcal{N}=2$} super
  {S}chr\"{o}dinger algebras and their representations,
 \emph{J. Math. Phys.}
 {\bf 58}, 113501, 14 (2017).

\bibitem{aku/dup/chi}
V.~P. Akulov, S.~Duplij, and V.~V. Chitov:
 Differential calculus for q-deformed twistors,
 \emph{Theor. Math. Phys.}
 {\bf 115}, 513--519 (1998)
 (\emph{Teor. Mat. Fiz.} \textbf{115},  177--184 (1998)).

\bibitem{and/ang/bag14}
N.~Andruskiewitsch, I.~Angiono, and D.~Bagio:
 Examples of pointed color {H}opf algebras,
 \emph{J. Algebra Appl.}
 {\bf 13}, 1350098, 28 (2014).


 \bibitem{brn/sch/sza}
 G.E.~Barnes, A.~Schenkel, and R.J.~Szabo:
Mapping spaces and automorphism groups of toric noncommutative spaces,
\emph{Lett. Math. Phys.} \textbf{107}, no. 9, 1591--1628 (2017).


\bibitem{leites13}
J.~Bernstein, D.~Leites, V.~Molotkov, and V.~Shander:
 \emph{Seminars of Supersymmetries. Vol.1. Algebra and calculus},
 MCCME,
 Moscow, 2013
 (In Russian, the English version is available for perusal).


 \bibitem{bra/lan/sui}
S.~Brain, G.~Landi, and W.D.~van Suijlekom:
Moduli spaces of instantons on toric noncommutative manifolds,
 \emph{Adv. Theor. Math. Phys.} \textbf{17}, no. 5, 1129--1193 (2013).


\bibitem{bru/iba/pon19}
A.~{Bruce}, E.~{Ibarguengoytia}, and N.~{Poncin}:
 {The Schwarz-Voronov embedding of $\mathbb{Z}_2^n$-manifolds},
   \emph{SIGMA Symmetry Integrability Geom. Methods Appl.} \textbf{16}, 002, 47 pages (2020).

\bibitem{bru/pon19b}
A.~Bruce and N.~Poncin:
Functional analytic issues in $\mathbb{Z}_2^n$-geometry,
 \emph{Rev. Uni\'{o}n Mat. Argent.} {\bf 60}(2), 611--636 (2019).

\bibitem{bru/pon19a}
A.~Bruce and N.~Poncin:
 Products in the category of {$\Bbb Z^n_2$}-manifolds,
 \emph{J. Nonlinear Math. Phys.}
 {\bf 26}, 420--453 (2019).

\bibitem{bru}
A.~J. Bruce:
 On a $\mathbb{Z}_2^n$-graded version of supersymmetry,
 \emph{Symmetry} {\bf 11}, 116 (2019).

\bibitem{bre93}
T.~Brzezi\'{n}ski:
 Remarks on bicovariant differential calculi and exterior {H}opf
  algebras,
 \emph{Lett. Math. Phys.}
 {\bf 27}, 287--300 (1993).


\bibitem{cel06}
S.~\c{C}elik:
Cartan calculi on the quantum superplane,
\emph{J. Math. Phys.} \textbf{47}, no. 8, 083501, 16 pp (2006).


\bibitem{cel16}
S.~\c{C}elik:
 Bicovariant differential calculus on the quantum superspace {$\Bbb
  R_q(1|2)$},
 \emph{J. Algebra Appl.}
 {\bf 15}, 1650172, 17 (2016).

\bibitem{cha/pre95}
V.~Chari and A.~Pressley:
 \emph{A guide to quantum groups},
 Cambridge University Press, Cambridge,
 1995 (Corrected reprint of the 1994 original).

 \bibitem{cir/lan/sza}
L.S.~Cirio,  G.~Landi, and R.J.~Szabo:
Algebraic deformations of toric varieties I. General constructions,
\emph{Adv. Math.} \textbf{246}, 33--88 (2013).

 \bibitem{con94}
A.~Connes:
 \emph{Noncommutative geometry},
 Academic Press, Inc., San Diego, CA,
 1994.

\bibitem{Cov/Gra/Pon1}
T.~Covolo, J.~Grabowski, and N.~Poncin:
 The category of {$\mathbb{Z}_2^n$}-supermanifolds,
 \emph{J. Math. Phys.}
 {\bf 57}, 073503, 16 (2016).

\bibitem{Cov/Gra/Pon2}
T.~Covolo, J.~Grabowski, and N.~Poncin:
 Splitting theorem for {$\mathbb{Z}_2^n$}-supermanifolds,
 \emph{J. Geom. Phys.}
 {\bf 110}, 393--401 (2016).

\bibitem{bo/gra/nie}
J.~{de Boer}, P.~A. {Grassi}, and P.~{van Nieuwenhuizen}:
 {Non-commutative superspace from string theory},
 \emph{Physics Letters B}
 {\bf 574}, 98--104 (2003).

\bibitem{deG:2015}
A.~de~Goursac:
 Noncommutative supergeometry and quantum supergroups,
 \emph{Journal of Physics: Conference Series}
 {\bf 597}, 012028 (2015).

\bibitem{del/fre99}
P.~Deligne and D.~S. Freed:
 Sign manifesto,
 in Quantum fields and strings: a course for mathematicians, {V}ol. 1, 2
  ({P}rinceton, {NJ}, 1996/1997),
 Amer. Math. Soc., Providence, RI,
 1999, pp.  357--363.


\bibitem{dri87}
V.~G. Drinfel'd:
 Quantum groups,
 in Proceedings of the {I}nternational {C}ongress of {M}athematicians, {V}ol.
  1, 2 ({B}erkeley, {C}alif., 1986) Amer. Math. Soc., Providence, RI, 1987, pp. 798--820.


\bibitem{dub}
   M.~Dubois-Violette:
 Generalized differential spaces with $d^N=0$ and the q-differential calculus, in  Quantum groups and integrable systems, I (Prague, 1996),
\emph{Czechoslovak J. Phys.} \textbf{46}, no. 12, 1227--1233 (1996).


\bibitem{duc/bla/hor/pen/sol}
G.~H.~E. {Duchamp}, P.~{Blasiak}, A.~{Horzela}, K.~A. {Penson}, and A.~I.
  {Solomon}:
{Hopf Algebras in General and in Combinatorial Physics: a practical
  introduction}, arXiv:0802.0249.


\bibitem{dup2018c}
S.~Duplij:
 Coderivations and codifferential calculi on finite groups,
 in Exotic Algebraic and Geometric Structures in Theoretical Physics,
  S.~Duplij, ed., Nova Publishers, New York, 2018, pp.  145--160.

\bibitem{dup/sie/bag}
S.~Duplij, W.~Siegel, and J.~Bagger, eds.:
 \emph{Concise Encyclopedia of Supersymmetry And Noncommutative Structures In
  Mathematics And Physics},
 Kluwer Academic Publishers,
 Dordrecht-Boston-London, 2004
 (Second printing, Springer Science and Business Media, Berlin-New
  York-Heidelberg, 2005).

\bibitem{dup/sin3}
S.~Duplij and S.~Sinel'shchikov:
 Classification of ${U}_{q}\left( \mathfrak{sl}_{2}\right) $-module
  algebra structures on the quantum plane,
 \emph{J. Math. Physics, Analysis, Geometry}
 {\bf 6}, 21--46 (2010).

\bibitem{fal/tah01}
M.~El~Falaki and E.~H. Tahri:
Quantum supergroup structure of {$(1+1)$}-dimensional quantum
  superplane, its dual and its differential calculus,
 \emph{J. Phys. A}
 {\bf 34}, 3403--3412  (2001).

\bibitem{fel01}
J.~Feldvoss:
 Representations of {L}ie colour algebras,
 \emph{Adv. Math.}
 {\bf 157}, 95--137 (2001).
 
 
  \bibitem{fil/isa/kur92}
A.~T. Filippov, A.~P Isaev and A.~B. Kurdikov:
 Para-Grassmann analysis and quantum groups,
\emph{Modern Phys. Lett. A} \textbf{7}, no. 23, 2129--2141 (1992). 
 
 \bibitem{fil/isa/kur93}
A.~T. Filippov, A.~P Isaev and A.~B. Kurdikov:
Para-Grassmann differential calculus,
\emph{Theoret. and Math. Phys.} \textbf{94}, no. 2, 150--165 (1993).  
 
 
 \bibitem{ger/tka}
 V.~D. Gershun and V.~I. Tkach: Para-grassman variables and description of massive particles with spin equalling one,
 \emph{Ukr. Fiz. Zh.} \textbf{29}, 1620 (1984). 
 
 \bibitem{ger/tka85}
 V.~D. Gershun and V.~I. Tkach: Description of arbitrary-spin particles on the base of local supersymmetry,   \emph{Probl. Nucl. Phys. Cosmic Rays} \textbf{23}, 42 (1985).
 

\bibitem{gra/var/fig01}
J.~M. Gracia-Bond\'{\i}a, J.~C. V\'{a}rilly, and H.~Figueroa:
 \emph{Elements of noncommutative geometry},
 Birkh\"{a}user Advanced Texts: Basler Lehrb\"{u}cher. [Birkh\"{a}user Advanced Texts: Basel Textbooks], Birkh\"{a}user Boston, Inc., Boston, MA,
 2001.

\bibitem{Gre}
H.~S. Green:
 A generalized method of field quantization,
 \emph{Phys. Rev.}
 {\bf 90}, 270--273 (1953).

\bibitem{hur/mak17}
B.~Hurle and A.~Makhlouf:
 Color {L}ie bialgebras: big bracket, cohomology and deformations,
 in Geometric and harmonic analysis on homogeneous spaces and applications,
  Vol. 207 of \emph{Springer Proc. Math. Stat.},
 Springer, Cham,
 2017,
 pp.  69--115.

 \bibitem{kap}
 M.M.~Kapranov:
 On the q-analog of homological algebra, arXiv:q-alg/9611005.


\bibitem{mad99}
J.~Madore:
 \emph{An introduction to noncommutative differential geometry and its physical applications},
 Vol. 257 of London Mathematical Society Lecture Note Series,
 Second edition,
 Cambridge University Press, Cambridge,
 1999.

 \bibitem{maj93}
 S.~Majid:
Beyond supersymmetry and quantum symmetry (an introduction to braided-groups and braided-matrices), Quantum groups, integrable statistical models and knot theory (Tianjin, 1992), 231--282,
Nankai Lectures Math. Phys., World Sci. Publ., River Edge, NJ, 1993.


\bibitem{maj95}
 S.~Majid:
\emph{Foundations of quantum group theory}, Cambridge University Press, Cambridge, 1995.

\bibitem{maj96}
 S.~Majid:
Introduction to braided geometry and q-Minkowski space, in Quantum groups and their applications in physics (Varenna, 1994), 267-345,
Proc. Internat. School Phys. Enrico Fermi, 127, IOS, \emph{Amsterdam}, 1996.


\bibitem{man89}
Y.~I. Manin:
 Multiparametric quantum deformation of the general linear supergroup,
 \emph{Comm. Math. Phys.}
 {\bf 123}, 163--175 (1989).

\bibitem{man91}
Y.~I. Manin:
 \emph{Topics in noncommutative geometry},
 M. B. Porter Lectures,
 Princeton University Press, Princeton, NJ,
 1991.

\bibitem{mon93}
S.~Montgomery:
 Hopf algebras and their actions on rings,
 Vol.~82 of \emph{CBMS Regional Conference Series in Mathematics},
 Published for the Conference Board of the Mathematical Sciences, Washington,
  DC; by the American Mathematical Society, Providence, RI,
 1993.

\bibitem{oza13}
M.~\"{O}zav\c{s}ar:
 A two-parameter quantum {$(2+1)$}-superspace and its deformed derivation
  algebra as {H}opf superalgebra,
 \emph{Adv. Appl. Clifford Algebr.}
 {\bf 23}, 741--756 (2013).

\bibitem{Sch}
M.~Scheunert:
Generalized {L}ie algebras,
\emph{J. Math. Phys.}
 {\bf 20}, 712--720  (1979).

\bibitem{sei03}
N.~Seiberg:
 Noncommutative superspace, {$\mathcal{N}=1/2$} supersymmetry, field
  theory and string theory,
\emph{J. High Energy Phys.} no. 6, 010, 17 (2003).

\bibitem{son91}
S.~K. Soni:
Differential calculus on the quantum superplane,
 \emph{J. Phys. A}
 {\bf 24}, 619--624 (1991).

\bibitem{tah98}
E.~H. Tahri:
Quantum group structure on the quantum plane,
 \emph{J. Math. Phys.}
 {\bf 39}, 2983--2992  (1998).

\bibitem{tol14}
V.~N. Tolstoy:
 Once more on parastatistics,
 \emph{Phys. Part. Nuclei Lett.}
 {\bf 11}, 933--937 (2014).

\bibitem{tol14a}
V.~N. Tolstoy:
 Super-de {S}itter and alternative super-{P}oincar\'e symmetries,
 in Lie Theory and Its Application in Physics,  V.~Dobrev, ed.,
 Springer, Tokyo, 2014, pp.  357--367.

\bibitem{Vol}
D.~V. Volkov:
 On the quantization of half-integer spin fields,
 \emph{Soviet Physics. JETP}
 {\bf 9}, 1107--1111 (1959).
 


\bibitem{wang11}
Y.-H. Wang:
 On graded global dimension of color {H}opf algebras,
 \emph{J. Gen. Lie Theory Appl.}
 {\bf 5}, Art. ID G110101, 6 (2011).

\bibitem{wes/zum91}
J.~Wess and B.~Zumino:
 Covariant differential calculus on the quantum hyperplane,
 \emph{Nuclear Phys. B Proc. Suppl.}
 {\bf 18B} (1990), 302--312 (1991).

\bibitem{woro89}
S.~L. Woronowicz:
 Differential calculus on compact matrix pseudogroups (quantum groups),
 \emph{Comm. Math. Phys.}
 {\bf 122}, 125--170  (1989).

\end{thebibliography}
\end{document}